\documentclass[11pt,reqno]{amsart}
\usepackage{amssymb}
\usepackage{amsmath}
\usepackage{graphicx}
\usepackage[latin1]{inputenc}
\usepackage{amsfonts}
\usepackage[spanish,english]{babel}
\usepackage{comma}
\usepackage[active]{srcltx}
\usepackage{url}
\usepackage{hyperref}

\newtheorem{theorem}{Theorem}
\theoremstyle{plain}

\newtheorem{corollary}[theorem]{Corollary}

\newtheorem{definition}[theorem]{Definition}
\newtheorem{example}[theorem]{Example}

\numberwithin{equation}{section}

\newcommand{\seqnum}[1]{\href{http://oeis.org/#1}{\underline{#1}}}

\DeclareMathOperator{\lcm}{lcm}

\begin{document}
\title[Overpseudoprimes and primover numbers]{Overpseudoprimes, and Mersenne and Fermat numbers as primover numbers}

\author[V. Shevelev et al.]{Vladimir Shevelev}
\address{Vladimir Shevelev, Department of Mathematics, Ben-Gurion University of the
 Negev, Beer-Sheva 84105, Israel}
\email{shevelev@bgu.ac.il}
\author[]{Gilberto Garc\'ia-Pulgar\'in}
\address{Gilberto Garc\'ia-Pulgar\'in, Universidad de Antioquia, Medellín-Colombia}
\email{gigarcia@ciencias.udea.edu.co}
\author[]{Juan Miguel Vel\'asquez-Soto}
\address{Juan Miguel Vel\'asquez-Soto, Departamento de Matemáticas, Universidad del Valle, Cali-Colombia}
\email{jumiveso@univalle.edu.co}
\author[]{John H. Castillo}
\address{John H. Castillo, Departamento de Matemáticas y Estadística, Universidad de Nariño, San Juan de Pasto-Colombia}
\email{jhcastillo@gmail.com}


\begin{abstract}
We introduce a new class of pseudoprimes-so called
``overpseudoprimes to base $b$", which is a subclass of strong
pseudoprimes to base $b$. Denoting via $|b|_n$ the multiplicative
order of $b$ modulo $n$, we show that a composite $n$ is
overpseudoprime if and only if $|b|_d$ is invariant for all divisors
$d>1$ of $n$. In particular, we prove that all composite Mersenne
numbers $2^{p}-1$, where $p$ is prime, are overpseudoprime to base
$2$ and squares of Wieferich primes are overpseudoprimes to base
$2$. Finally, we show that some kinds of well known numbers are
overpseudoprime to a base $b$.
\end{abstract}
\maketitle
\section{Introduction}
First and foremost, we recall some definitions and fix some
notation. Let $b$ an integer greater than $1$ and $N$ a positive
integer relatively prime to $b$. Throughout, we denote by $|b|_N$
the multiplicative order of $b$ modulo $N$. For a prime $p$,
$\nu_{p}(N)$ means the greatest exponent of $p$ in the prime
factorization of $N$.

Fermat's little theorem implies that $2^{p-1}\equiv 1 \pmod {p}$,
where $p$ is an odd prime $p$. An odd prime $p$, is called a
Wieferich prime if $2^{p-1}\equiv 1\pmod{p^2}$,

We recall that a Poulet number, also known as Fermat pseudoprime to
base $2$, is a composite number $n$ such that $2^{n-1}\equiv 1
\pmod{n}$. A Poulet number $n$  which verifies that $d$ divides
$2^d-2$ for each divisor $d$ of $n$, is called a Super-Poulet
pseudoprime.

Sometimes the numbers $M_n=2^n-1, \enskip n=1,2,\ldots $, are called
Mersenne numbers, although this name is usually reserved for numbers
of the form

\begin{equation}\label{1}
M_p= 2^p -1
\end{equation}
where $p$ is prime. In this form numbers $M_p$, at the first time,
    were studied by Marin Mersenne (1588-1648)  around 1644; see Guy
\cite[\S A3]{GUY04} and a large bibliography there.

In the next section, we introduce a new class of pseudoprimes and we
prove that it just contains the odd numbers $n$ such that $|2|_d$ is
invariant for all divisors greater than $1$ of $n$. In particular,
we show that it contains all composite Mersenne numbers and, at
least, squares of all Wieferich primes. In the fourth section, we
give a generalization of this concept to arbitrary bases $b>1$ as
well. In the final section, we put forward some of its consequences.

We note that, the concept of overpseudoprime to base $b$ was found
in two independent ways. The first one in 2008, by Shevelev
\cite{overpseudoprimes} and the second one, by Castillo et al.
\cite{midypseudoprimes}, using consequences of the Midy's property,
where overpseudoprimes numbers are denominated Midy pseudoprimes.

The first sections of the present work is a revisited version of
Shevelev \cite{overpseudoprimes}. In the last section, we present a
review of Shevelev \cite{primoverization}, using results from
Castillo et al. \cite{midypseudoprimes}.

The sequences \seqnum{A141232}, \seqnum{A141350} and
\seqnum{A141390} in \cite{SEQ}, are result of the earlier work of
Shevelev.

\section{A class of pseudoprimes}

Let $n>1$ be an odd number. When we multiply by $2$ the set of
integers modulo $n$, we split it in different sets called
\emph{cyclotomic cosets}. The cyclotomic coset containing $s\neq 0$
consists of $C_s=\{s,2s,2^2s,\ldots,2^{m_s-1}s\}$, where $m_s$ is
the smallest positive number such that $2^{m_s}\cdot s\equiv s\pmod
n$. Actually, it is easy to see that
$m_s=|2|_{\frac{n}{\gcd(n,s)}}$. For instance the cyclotomic cosets
modulo $15$ are

\begin{align*}
C_1&=\{1,2,4,8\},\\
C_3&=\{3,6,12,9\},\\
C_5&=\{5,10\}, \text{ and }\\
C_7&=\{7,14,13,11\}.
\end{align*}

Denote by $r=r(n)$, the number of distinct cyclotomic cosets of $2$
modulo $n$. From the above example, $r(15)=4$.

Note that, if $C_1,\ldots, C_r$ are the different cyclotomic cosets
of $2$ modulo $n$, then
\begin{equation}\label{2}
\bigcup^r_{j=1}C_j=\{1,2,\ldots, n-1\} \text{ and } C_{j_1}\cap
C_{j_2}=\varnothing , \;\; j_1\neq j_2.
\end{equation}

We can demonstrate that

\begin{equation}\label{3}
|2|_n=\lcm(|C_1|,\ldots,|C_r|).
\end {equation}

If $p$ is an odd prime the cyclotomic cosets have the same number of
elements, because for each $s\neq 0$ we have
$m_s=|C_s|=|2|_{\frac{p}{\gcd(p,s)}}=|2|_p$. So
\begin{equation}\label{4}
|C_1|=\cdots=|C_r|.
\end{equation}
Therefore, when $p$ is an odd prime, we obtain
\begin{equation}\label{5}
p= r(p)|2|_p + 1.
\end{equation}
This leave us to study composite numbers such that the equation
\eqref{5} holds.

\begin{definition}
We say that  an odd composite number $n$ is an overpseudoprime to
base $2$, if
\begin{equation}\label{6}
n=r(n)|2|_n+1.
\end {equation}
\end{definition}

Note that if $n$ is an overpseudoprime to base $2$, then $2^{n-1}=
2^{r(n)|2|_n}\equiv 1\pmod{n}$. Thus, the set of overpseudoprimes to
base $2$ is a subset of the set of Poulet pseudoprimes to base $2$.

\begin{theorem}\label{t2}

Let $n=p_1^{l_1}\cdots p_k^{l_k}$ be an odd composite number. Then
$n$ is an overpseudoprime to base $2$ if and only if
\begin{equation}\label{8} |2|_n=|2|_{d},
\end{equation} for each divisor $d>1$
of $n$.

\end{theorem}
\begin{proof}
Let $s$, different from zero, be an arbitrary element of $\mathbb{Z}_{n}$. Take $u_{s}=\gcd( n, s) $ and $v_{s}=%
\dfrac{n}{u_{s}}$. Then $s=au_{s}$, for some integer $a$ relatively prime with $%
n$. As we said before, $\left\vert C_{s}\right\vert =\left\vert
2\right\vert _{v_{s}}$.

Note that when $s$ runs through a set of coset representatives
modulo $n$, $v_s$ runs through the set of divisors of $n$. So the
value of $\left\vert C_{s}\right\vert $ is constant if and only if
$\left\vert 2\right\vert _{d}$ is invariant for each divisor $d>1$
of $n$, which proves the theorem.
\end{proof}

A direct consequence of the last theorem is the following.

\begin{corollary}\label{cor3} Two overpseudoprimes to base $2$,
 $N_1$ and $N_2$ such that $|2|_{N_1}\neq|2|_{N_2}$, are relatively primes.
 \end{corollary}

\begin{corollary}\label{cor2}
For a prime $p$, $M_p=2^p-1$ is either a prime or an overpseudoprime
to base $2$.
 \end{corollary}
\begin{proof}
Assume that $M_p$ is not prime. Let $d>1$ be any divisor of $M_p$.
Then $|2|_d$ divides $p$ and thus $|2|_d=p$.
\end{proof}

\begin{corollary}\label{cor5} Every overpseudoprime to base $2$ is a Super-Poulet pseudoprime.
\end{corollary}
\begin{proof}
Let $n$ be an overpseudoprime to base $2$ and take $d$ an arbitrary
divisor of $n$. By Theorem \ref{t2}, $d$ is either prime or
overpseudoprime to base $2$. In any case, we have $2^{d-1}\equiv 1
\pmod {d}$.
\end{proof}

\begin{example} Consider the super-Poulet pseudoprime, see
\seqnum{A178997} in \cite{SEQ}, $96916279 = 167 \cdot 499 \cdot
1163$. We know that, cf. \seqnum{A002326} in \cite{SEQ},  $
|2|_{167}=83, \; |2|_{499}=166$ and $|2|_{1163} =166$. Thus the
reciprocal of the above corollary is not true.
\end{example}

Assume that $p_1$ and $p_2$ are primes such that $|2|_{p_1}=
|2|_{p_2}$. Then $|2|_{p_1p_2}= \lcm(|2|_{p_1},|2|_{p_2})$. In
consequence, $n=p_1p_2$ is an overpseudoprime to base $2$. With the
same objective, we get the following.
\begin{theorem}\label{t3}Let $p_1,\ldots,p_k$ be different primes such that $|2|_{p_i}=|2|_{p_j}$, when $i\neq j$.
Assume that $p_i^{l_i}$ is an overpseudoprime to base $2$, where
$l_i$ are positive integers, for each $i=1,\ldots,k$. Then
$n=p_1^{l_1}\cdots p_k^{l_k}$ is an overpseudoprime to base $2$.
\end{theorem}

\section[Wieferich primes]{The $(w+1)$-th power of Wieferich prime of order $w$ is overpseudoprime to base $2$}
Knauer and Richstein \cite{KR05}, proved that  $1093$ and $3511$ are
the only Wieferich primes less than $1.25\times 10^{15}$. More
recently, Dorais and Klyve \cite{DK11} extend this interval to $6.7
× 10^{15}$.

We say that a prime $p$ is a Wieferich prime of order $w\geq 1$, if
$\nu_p(2^{p-1}-1)=w+1$.

The following result, from Nathanson \cite[Thm.\ 3.6]{Nathanson},
give us a method to calculate $|b|_{p^t}$ from $|b|_p$.

\begin{theorem}
\label{poten}Let $p$ be an odd prime not divisor of $b$, $m=\nu_{p}%
(b^{\left\vert b\right\vert _{p}}-1)$ and $t$  a positive integer,
then
\[
\left\vert b\right\vert _{p^{t}}=%
\begin{cases}
\left\vert b\right\vert _{p}, & \text{ if }t\leq m;\\
& \\
p^{t-m}\left\vert b\right\vert _{p}, & \text{ if \ }t>m.
\end{cases}
\]

\end{theorem}

\begin{theorem}\label{t4}
A prime $p$ is a Wieferich prime of order greater than or equal to
$w$ if and only if $p^{w+1}$ is an overpseudoprime to base $2$.
\end{theorem}
\begin{proof}
Suppose that $p$ is a Wieferich prime of order greater than or equal
to $w$. Then $p^{w+1}\mid 2^{p-1}-1$ and thus $\left\vert
2\right\vert _{p^{w+1}}$ is a divisor of $p-1$.

By Theorem \ref{poten}, $\left\vert 2\right\vert _{p^{w+1}}=p^{r}$ $%
\left\vert 2\right\vert _{p}$ for some non-negative integer $r$. So,
$r=0$. Therefore, $p^{w+1}$ is an overpseudoprime to base $2$.  The
reciprocal is clear.

\end{proof}

\begin{theorem}\label{t5}
Let $n$ be an overpseudoprime to base $2$. If $n$ is not the
multiple of the square of a Wieferich prime, then $n$ is squarefree.
\end{theorem}

\begin{proof}
Let $n=p_1^{l_1} \ldots p_k^{l_k}$ and, say, $l_1\geq 2$. If $p_1$
is not a Wieferich prime, then $|2|_{p_1^2}$ divides $p_1(p_1-1)$
but does not divide $p_1-1$. Thus, $|2|_{p_1^2}\geq p_1$. Since
$|2|_{p_1}\leq p_1-1$, then $|2|_{p_1^2} > |2|_{p_1}$ and by Theorem
\ref{t2}, $n$ is not an overpseudoprime to base $2$.
\end{proof}

\section[Overpseudoprime to base b]{Overpseudoprime to base $b$}
Take $b$ a positive integer greater than $1$. Denote by $r=r_b(n)$
the number of cyclotomic cosets of $b$ modulo $n$. If
$C_1,\ldots,C_r$ are the different cyclotomic cosets of $b$ modulo
$n$, then $C_{j_1} \cap C_{j_2}=\varnothing , \;\; j_1\neq j_2$ and
$\bigcup^r_{j=1}C_j=\{1,2,\ldots, n-1\}$.

 Let $p$ be a prime which does not divide $b(b-1)$. Once again, we get
$r_b(p)|b|_p=p-1.$

\begin{definition}
We say that a composite number $n$, relatively prime to $b$, is an
overpseudoprime to base $b$, if it satisfies
\begin{equation}\label{16}
n=r_b(n)|b|_n+1.
\end{equation}
\end{definition}

The proof of the next theorem follows similarly as in Theorem
\ref{t2}.
\begin{theorem}\label{t7}
Let $n$ be a composite number such that $\gcd(n,b)=1$. Then $n$ is
an overpseudoprime to base $b$ if and only if  $|b|_{n}=|b|_{d}$,
for each divisor $d>1$ of $n$.
\end{theorem}

\begin{definition}
A prime $p$ is said a Wieferich prime in base $b$ if $ b^{p-1}\equiv
1\pmod {p^2}$. A Wieferich prime to base $b$ is of order $w\geq 1$,
if $\nu_p(b^{p-1}-1)=w+1$.
\end{definition}

With this definition in our hands, we can generalize Theorems
\ref{t4} and \ref{t5}. The respective proofs, are similar to that
ones.
\begin{theorem}
\label{t10} A prime $p$ is a Wieferich prime in base $b$ of order
greater than or equal to $w$ if and only if $p^{w+1}$ is an
overpseudoprime to base $b$.
\end{theorem}
\begin{theorem}\label{t11}
If $n$ is overpseudoprime to base $b$ and is not a multiple of a
square of a Wieferich prime to base $b$, then $n$ is squarefree.

\end{theorem}

Let us remember that an odd composite $N$ such that $N-1=2^{r}s$
with $s$ an odd integer and $\left(  b,\ N\right)  =1$, is a strong
pseudoprime to base $b$~ if either $b^{s}\equiv 1\ \pmod N$ or
$b^{2^{i}s}\equiv-1\ \pmod N$, for some $0\leq i<r$. The following
result shows us, that the overpseudoprimes do not appear more
frequently than the strong pseudoprimes.

\begin{theorem}\label{t12}
If $n$ is an overpseudoprime to base $b$, then $n$ is a strong
pseudoprime to the same base.
\end{theorem}
\begin{proof}
Let $n$ be an overpseudoprime to base $b$. Suppose that $n-1=2^{r}s$
and $|b|_{n}=2^{t}s_{1}$, for some odd integer $s$, $s_{1}$ and
nonnegative integers $r$, $t$. Since $n$ is an overpseudoprime, then
$|b|_{n}|n-1$. Thus $t\leq r$ and $s_1$ divides $s$. Assume $t=0$.
So $|b|_{n}$ is a divisor of $s$ and thus
\[
b^{s}\equiv1 \pmod{n}.
\]
Then $n$ is a strong pseudoprime to base $b$.

On the other side, assume that $t\geq1$ and write $A=b^{s_{1}}=b^{\frac{|b|_{n}}{2^{t}}%
}$. Note that
\[
(A-1)(A+1)(A^{2}+1)(A^{2^{2}}+1)\cdots(A^{2^{t-1}}+1)=A^{2^{t}}-1\equiv
0 \pmod{n}.
\]

We claim that for any $i<t-1$ the greatest common divisor $\gcd( n,\
A^{2^{i}}+1)$ is $1$ . Indeed, assume that $d>1$ divides both $n$
and $A^{2^{i}}+1$. Since $n$ is an overpseudoprime to base $b$, we
have $|b|_{d}=|b|_{n}$ and the congruence
$A^{2^{i}}=b^{2^{i}s_{1}}\equiv -1\pmod{d}$, leave us to a
contradiction with the definition of $|b|_{d}$. Thus,
$\gcd(A^{2^{i}}+1,n)=1$. Similarly $\gcd\left(  A-1,\ n\right)  =1$
and we obtain
\[
A^{2^{t-1}}+1\equiv0\pmod{n}.
\]

Consequently, $b^{2^{t-1}s}\equiv-1\pmod{n}$. Therefore, $n$ is a
strong pseudoprime to base $b$.
\end{proof}

Note that there are strong pseudoprimes to base $b$ such that $|b|_{n}%
=2^{t}s_{1}$ and $b^{2^{i}s_{1}}\not \equiv-1\pmod{n}$ for $i<t-1$,
but $n$ is not an overpseudoprime to base $b$. For example
$n=74415361$ and $b=13$.


As before, where we have proved that every overpseudoprime to base
$2$ is super-Poulet pseudoprime, using Theorem \ref{t7} we can prove
the following statement.
\begin{theorem}\label{t13}

Every overpseudoprime $n$ to base $b$ is a superpseudoprime, that is
\begin{equation}\label{20} b^{d-1}\equiv 1\pmod{d},
\end{equation}
for each divisor $d>1$ of $n$.
\end{theorem}
\begin{theorem}\label{t14}
If $n$ is an overpseudoprime to base $b$, then for every two
divisors $d_1 < d_2$ of  $n$, including  $1$ and $n$, we have
\begin{equation}\label{21}
|b|_n |d_2-d_1.
\end{equation}
\end{theorem}
\begin{proof}
 By the equation \eqref{20}, we have
$|b|_{d_i}=|b|_n$ divides $d_i-1$, for $i=1,2,$ and thus \eqref{21}
follows.
\end{proof}

\section{Primoverization Process}

Note that, if $n$ is an overpseudoprime to base $b$, a divisor of
$n$ is either prime or overpseudoprime to base $b$.  In this section
we study some kinds of numbers which satisfy this property.

In the sequel, we denote by $\Phi_{n}\left( x\right) $ the $n$-th
cyclotomic polynomial. We recall the following theorems from
Castillo et al. \cite{midypseudoprimes}.

\begin{theorem}
A composite number $N$ with $\gcd\left(  N,\left\vert b\right\vert
_{N}\right) =1,$ is an overpseudoprime to base $b$ if and only if
$\Phi_{\left\vert b\right\vert _{N}}\left(  b\right)  \equiv 0\pmod
N$ and $\left\vert b\right\vert _{N}>1$.
\end{theorem}

\begin{theorem}
\label{primover}Let $N>2$ and $P_{N}\left(  b\right)
=\dfrac{\Phi_{N}\left( b\right)  }{\gcd\left(  N,\ \Phi_{N}\left(
b\right)  \right)  }$. If $P_{N}\left(  b\right)  $ is composite,
then $P_{N}\left(  b\right)  $ is an overpseudoprime to base $b$.
\end{theorem}
 The last theorem leave us to the next definition.

 \begin{definition}
 A positive integer is called primover to base $b$ if it is either
 prime or an overpseudoprime to base $b$.
 \end{definition}

 By Theorem \ref{t7}, we know that each divisor greater than $1$, of
 a overpseudoprime to base $b$ is primover to the same base $b$. By Corollary \ref{2}, $M_p$ is primover to base $2$.

Theorem \ref{primover} suggests that we need to know the value of
$\gcd\left(  N,\ \Phi_{N}\left( b\right)  \right)$. To that
objective, we recall a result from Motose \cite[Th.\ 2]{Motose2}.

\begin{theorem}\label{th2Motose}
We set $n\geq 2$, $a\geq 2$. Then $p$ is a prime divisor of
$\Phi_n(b)$ if and only if $\gcd(b,p)=1$ and $n=p^{\gamma}|b|_p$
where $\gamma\geq 0$. A prime divisor $p$ of $\Phi_n(b)$ for $n\geq
3$ has the property such that $n=|a|_p$ or $\nu_p(\Phi_n(b))=1$ as
$\gamma=0$ or not.
\end{theorem}

Let $p$ be the greatest prime divisor of $N$. We claim that either
$\gcd(N,\Phi_N(b))=1$ or $p$. Indeed, assume that there is a prime
$q<p$ divisor of $N$ and $\Phi_N(b)$. Thus, Theorem \ref{th2Motose}
implies that $N=q^{\gamma}|b|_q$. But as $p$ divides $N$, we obtain
a contradiction. So $\gcd(N,\Phi_N(b))$, is either $1$ or a power of
$p$. If $\gcd(N,\Phi_N(b))>1$, then $N=p^l|b|_p$. Since $l>0$,
Theorem \ref{th2Motose} implies that $p^2$ does not divide
$\Phi_N(b)$. Therefore, we get the following corollary.

\begin{corollary}\label{th7Trio}
Let $N>1$ and $p$ the greatest prime divisor of $N$. Then
$\gcd(N,\Phi_N(b))=1$ or $p$.
\end{corollary}

In the sequel, we prove that some known kinds of numbers are
primovers to some base $b$.

 \begin{theorem}\label{fermatnumbers}
 A generalized Fermat number, $F_n(b)=b^{2^n}+1$, with $n$ a positive integer and $b$ even;
 is primover to base $b$.
 \end{theorem}
 \begin{proof}
It is well known that if $p$ is prime, then $\Phi_{p^{r}}\left(
x\right)=\dfrac{x^{p^{r}}-1}{x^{p^{r-1}}-1}$, see Bamunoba
\cite[Thm. \ 3.4.6]{Bam} or Gallot \cite[Thm. 1.1]{GAL}. Since
$\gcd\left(  2^{n+1},\
\Phi_{2^{n+1}}\left( b\right)  \right)  =1$, we have $P_{2^{n+1}%
}(b)=F_{n}\left(  b\right) $  and the result follows from Theorem
\ref{primover}.
 \end{proof}

 \begin{theorem}\label{mersennenumber}
A generalized Mersenne number, $M_{p}\left( b\right) =\dfrac
{b^{p}-1}{b-1}$, with $p$  a prime such that $\gcd(p,b-1)=1$, is
primover to base $b$.
\end{theorem}

\begin{proof}
Note that $\Phi_p(b)=M_{p}\left( b\right)$ and
$\gcd(p,\Phi_p(b))=1$. So $P_p(b)=M_p(b)$ and the result follows
from Theorem \ref{primover}.
\end{proof}

By Theorems \ref{t14} and \ref{mersennenumber}, once again, we can
prove that the numbers $M_p{(b)}$ satisfy a similar property of the
Mersenne numbers $M_p$.
\begin{corollary}\label{cor4}
If ~$\gcd(p,b-1)=1$, then for every pair of divisors $d_1 < d_2$ of
$M_p{(b)}$, including trivial divisors $1$ and $M_p{(b)}$, we have
\begin{equation}\label{22}
 p| d_2-d_1.
\end{equation}
\end{corollary}
The following corollary give us an interesting property of $M_r(b)$.
\begin{corollary} Let $r$ be a prime with $\gcd(r,b-1)=1$. Then $M_r(b)$ is prime if and only if the progression $(1+rx)_{x\geq0}$
contains just one prime $p$ such that $|b|_p=r$.
\end{corollary}

\begin{proof}
Assume that $M_r(b)$ is prime. If there exists a prime $p$, such
that $|b|_p=r$, then $p=M_r(b)$. Since $r|p-1$, i.e., $p$ is the
unique prime in the progression $(1+rx)_{x\geq 0}$.

Conversely, assume that there exists only one prime of the form
$p=1+rx$, with $x\geq 0$, such that $|b|_p=r$. So $p$ divides
$M_r(b)$. If $M_r(b)$ is composite, then it is overpseudoprime to
base $b$ and thus to other prime divisor $q$ of $M_r(b)$ we obtain
$|b|_q=r$. This contradicts our assumption.
\end{proof}

The next result shows that Fermat numbers to base $2$ are the only
ones, of the form $2^m+1$, which are primover to base $2$.

\begin{theorem}
The following properties hold.
\begin{enumerate}
\item  Assume that $b$ is even. Then $P_m(b)=b^{m}+1$
is primover to base $b$ if and only if $m$ is a power of $2$.

\item Suppose that $\gcd(n,b-1)=1$. Then $M_n(b)=\dfrac{b^{n}-1}{b-1}$ is primover to base $b$ if and only if  $n$
 is prime.
\end{enumerate}
\end{theorem}

\begin{proof}
Sufficient conditions were proved in Theorems \ref{fermatnumbers}
and \ref{mersennenumber}.

Now assume that  $m$ has an odd prime divisor. So $b+1$ is a divisor
of $P_m(b)$ and thereby it is not a prime. Since, $\left\vert
b\right\vert _{b+1}=2$ and $\left\vert b \right\vert _{b^m+1}=2m$;
also it is not an overpseudoprime to base $b$.

To prove the necessity of the second part, suppose that $n$ is not
prime. Thus for a prime $p$ divisor of $n$, we have $M_{n}(b)$ is
composite and $b^{p}-1$ is one of its proper divisors. As
$\left\vert b\right\vert _{b^{p}-1}=p$ and $\left\vert b\right\vert
_{M_{n}(b)}=n$, we get that $M_{n}(b)$ is not an overpseudoprime to
base $b.$
\end{proof}

We note that, for $p$ and $q$ primes with $q<p$,
$|b|_{\Phi_{pq}(b)}=pq$.

\begin{theorem}
\label{pq}If $q<p$ are primes, then
\[
N=\frac{(b-1)(b^{pq}-1)}{(b^{p}-1)(b^{q}-1)}%
\]
is primover to base $b$ if and only if $N$ is not multiple of $p$.
\end{theorem}

\begin{proof}
It is clear that, $N=\Phi_{pq}\left(b\right)$. Assume that $N$ is
not a multiple of $p$. Corollary \ref{th7Trio} implies that
$\gcd\left( pq,\ \Phi_{pq}\left( b\right) \right) =1$ and the result
follows from Theorem \ref{primover}.

Conversely assume that $N$ is primover to base $b$ and $p$ divides
$N$. Thereby, $|b|_p$ divides $q$ and as $|b|_N=pq$, we get a
contradiction.
\end{proof}

\begin{corollary}
With the above notation, if $p$ divides $N$, then $\dfrac {N}{p}$ is
primover to base $b$.
\end{corollary}

Once again, using Corollary \ref{th7Trio} and Theorem \ref{primover}
we can prove the following theorems.
\begin{theorem}
If $p$ is prime, then
\[
N=\frac{b^{p^{n}}-1}{b^{p^{n-1}}-1}%
\]
is primover to base $b$ if and only if $N$ is not multiple of $p$.
\end{theorem}

\begin{theorem}
Let $n=p_{1}p_{2}\cdots p_{t}$, where  $p_{1}<p_{2}<\cdots<p_{t}$
are primes and let
\[
N=\prod\nolimits_{e|n}\left(  b^{e}-1\right)^{\mu\left(  e\right)
\mu\left( n\right)}.
\]
If $\gcd(N,p_t)=1$, then $N$  is primover to base $b$. In other
case, $\dfrac{N}{p_{t}}$ is  primover to base $b$.
\end{theorem}


\begin{thebibliography}{99}

\bibitem{Bam}
A.~S.~Bamunoba. Cyclotomic polynomials. Thesis master of science in
the African Institute for Mathematical Sciences. Stellenbosch
University, South Africa,
\url{http://users.aims.ac.za/~bamunoba/bamunoba.pdf} (2010).

\bibitem{midypseudoprimes}
J.~H.~Castillo, G.~Garc{\'{\i}}a-Pulgar{\'{\i}}n, and J.~M.~
  Vel{\'a}squez-Soto, Pseudoprimes stronger than strong
  pseudoprimes, preprint,
  arXiv:1202.3428v2 [math.NT] (2012). (Manuscript submitted for
  publication)

\bibitem{DK11}
F.~G.~Dorais and D.~Klyve, A {W}ieferich prime search up
  to $6.7\times 10\sp {15}$, \emph{J. Integer Seq.} \textbf{14} (2011),
  Article 11.9.2.

\bibitem{GAL}

Y. Gallot, Cyclotomic polynomials and prime numbers, preprint,
\url{http://yves.gallot.pagesperso-orange.fr/papers/cyclotomic.pdf}

\bibitem{GUY04}
R.~K.~Guy, \emph{Unsolved Problems in Number Theory}, third ed.,
Problem Books in Mathematics, Springer-Verlag, 2004.

\bibitem{KR05}
J.~Knauer and J.~Richstein, The continuing search for {W}ieferich
primes, \emph{Math. Comp.} \textbf{74} (2005), no.~251, 1559--1563
(electronic).

\bibitem{Motose2}
K.~Motose, On values of cyclotomic polynomials. {II}, \emph{Math. J.
Okayama Univ.} \textbf{37} (1995), 27--36 (1996).


\bibitem{Nathanson}
M.~B.~Nathanson, \emph{Elementary Methods in Number Theory},
Graduate Texts in Mathematics, vol. 195, Springer-Verlag, 2000.

\bibitem{overpseudoprimes}
V.~Shevelev, Overpseudoprimes, Mersenne Numbers and Wieferich
  primes, preprint, arXiv:0806.3412v7 [math.NT] (2008).

\bibitem{primoverization}
V.~Shevelev, Process of primoverization of numbers of the form
$a^n-1$, preprint, arXiv:0807.2332v2 [math.NT] (2008).

\bibitem{SEQ}
N.~J.~A.~Sloane, The on-line encyclopedia of integer sequences,
published electronically at \url{http://oeis.org}.

\end{thebibliography}
\end{document}